\documentclass[a4paper, 12pt]{amsart}

\usepackage{amsmath,amssymb,enumitem,verbatim,stmaryrd,xcolor,microtype,graphicx}
\usepackage[T1]{fontenc}
\usepackage[utf8]{inputenc}
\usepackage[english]{babel}
\usepackage[top=3.4cm,bottom=3.4cm,left=3.2cm,right=3.2cm]{geometry}
\usepackage[bookmarksdepth=3,linktoc=page,colorlinks,linkcolor={blue!80!black},citecolor={blue!80!black},urlcolor={blue!80!black}]{hyperref}

\usepackage{tikz}\usetikzlibrary{matrix,arrows,decorations.markings}
\usepackage{tikz-cd}
\pgfrealjobname{normtrace}

\usepackage[mathcal]{euscript}                               
\usepackage{mathptmx}                                        
\usepackage{etoolbox}\makeatletter\patchcmd{\@startsection}{\@afterindenttrue}{\@afterindentfalse}{}{}\makeatother    
\patchcmd{\section}{\scshape}{\bfseries}{}{}\makeatletter\renewcommand{\@secnumfont}{\bfseries}\makeatother           
\usepackage[backgroundcolor=yellow,linecolor=yellow,textsize=footnotesize]{todonotes} \makeatletter \providecommand \@dotsep{5} \def\listtodoname{List of Todos} \def\listoftodos{\@starttoc{tdo}\listtodoname} \makeatother 

\theoremstyle{plain}
\newtheorem{thm}{Theorem}[section]
\newtheorem{cor}[thm]{Corollary}
\newtheorem{lemma}[thm]{Lemma}
\newtheorem{prop}[thm]{Proposition}

\theoremstyle{definition}

\newtheorem{rem}[thm]{Remark}
\newtheorem*{rem*}{Remark}

\usepackage{etoolbox}
\makeatletter
\patchcmd{\@startsection}{\@afterindenttrue}{\@afterindentfalse}{}{}             
\patchcmd{\part}{\bfseries}{\bfseries\LARGE}{}{}
\patchcmd{\section}{\scshape}{\bfseries}{}{}\renewcommand{\@secnumfont}{\bfseries} 
\patchcmd{\@settitle}{\uppercasenonmath\@title}{\large}{}{}
\patchcmd{\@setauthors}{\MakeUppercase}{}{}{}
\addto{\captionsenglish}{} 
\addto{\captionsenglish}{} 
\addto{\captionsenglish}{} 
\makeatother

\usepackage{fancyhdr}

\DeclareRobustCommand{\gobblefour}[4]{}    
\DeclareSymbolFont{sfoperators}{OT1}{bch}{m}{n} \DeclareSymbolFontAlphabet{\mathsf}{sfoperators} \makeatletter\def\operator@font{\mathgroup\symsfoperators}\makeatother 
\DeclareSymbolFont{cmletters}{OML}{cmm}{m}{it}              
\DeclareSymbolFont{cmsymbols}{OMS}{cmsy}{m}{n}
\DeclareSymbolFont{cmlargesymbols}{OMX}{cmex}{m}{n}
\DeclareMathSymbol{\myjmath}{\mathord}{cmletters}{"7C}     \let\jmath\myjmath 
\DeclareMathSymbol{\myamalg}{\mathbin}{cmsymbols}{"71}     
\DeclareMathSymbol{\mycoprod}{\mathop}{cmlargesymbols}{"60}
\DeclareMathSymbol{\myalpha}{\mathord}{cmletters}{"0B}     \let\alpha\myalpha 
\DeclareMathSymbol{\mybeta}{\mathord}{cmletters}{"0C}      \let\beta\mybeta
\DeclareMathSymbol{\mygamma}{\mathord}{cmletters}{"0D}     \let\gamma\mygamma
\DeclareMathSymbol{\mydelta}{\mathord}{cmletters}{"0E}     \let\delta\mydelta
\DeclareMathSymbol{\myepsilon}{\mathord}{cmletters}{"0F}   \let\epsilon\myepsilon
\DeclareMathSymbol{\myzeta}{\mathord}{cmletters}{"10}      \let\zeta\myzeta
\DeclareMathSymbol{\myeta}{\mathord}{cmletters}{"11}       \let\eta\myeta
\DeclareMathSymbol{\mytheta}{\mathord}{cmletters}{"12}     \let\theta\mytheta
\DeclareMathSymbol{\myiota}{\mathord}{cmletters}{"13}      \let\iota\myiota
\DeclareMathSymbol{\mykappa}{\mathord}{cmletters}{"14}     \let\kappa\mykappa
\DeclareMathSymbol{\mylambda}{\mathord}{cmletters}{"15}    \let\lambda\mylambda
\DeclareMathSymbol{\mymu}{\mathord}{cmletters}{"16}        \let\mu\mymu
\DeclareMathSymbol{\mynu}{\mathord}{cmletters}{"17}        \let\nu\mynu
\DeclareMathSymbol{\myxi}{\mathord}{cmletters}{"18}        \let\xi\myxi
\DeclareMathSymbol{\mypi}{\mathord}{cmletters}{"19}        \let\pi\mypi
\DeclareMathSymbol{\myrho}{\mathord}{cmletters}{"1A}       \let\rho\myrho
\DeclareMathSymbol{\mysigma}{\mathord}{cmletters}{"1B}     \let\sigma\mysigma
\DeclareMathSymbol{\mytau}{\mathord}{cmletters}{"1C}       \let\tau\mytau
\DeclareMathSymbol{\myupsilon}{\mathord}{cmletters}{"1D}   \let\upsilon\myupsilon
\DeclareMathSymbol{\myphi}{\mathord}{cmletters}{"1E}       \let\phi\myphi
\DeclareMathSymbol{\mychi}{\mathord}{cmletters}{"1F}       \let\chi\mychi
\DeclareMathSymbol{\mypsi}{\mathord}{cmletters}{"20}       \let\psi\mypsi
\DeclareMathSymbol{\myomega}{\mathord}{cmletters}{"21}     \let\omega\myomega
\DeclareMathSymbol{\myvarepsilon}{\mathord}{cmletters}{"22}\let\varepsilon\myvarepsilon
\DeclareMathSymbol{\myvartheta}{\mathord}{cmletters}{"23}  \let\vartheta\myvartheta
\DeclareMathSymbol{\myvarpi}{\mathord}{cmletters}{"24}     \let\varpi\myvarpi
\DeclareMathSymbol{\myvarrho}{\mathord}{cmletters}{"25}    \let\varrho\myvarrho
\DeclareMathSymbol{\myvarsigma}{\mathord}{cmletters}{"26}  \let\varsigma\myvarsigma
\DeclareMathSymbol{\myvarphi}{\mathord}{cmletters}{"27}    \let\varphi\myvarphi

\DeclareMathOperator{\Norm}{{Norm}}
\DeclareMathOperator{\Tr}{{Tr}}

\newcommand\F{{\mathbb F}}

\newcommand\Fq{{\mathbb F}_q}
\newcommand\Fqn{{\mathbb F}_{q^n}}

\renewcommand\int{\textup{int}}

\newcommand\ord{\textup{ord}}

\renewcommand\geq{\geqslant}
\renewcommand\leq{\leqslant}

\title{On the number of elements with prescribed norm and trace}

\author{Roberto Alvarenga}
\address{\rm Roberto Alvarenga, Instituto de Ci\^encias Matem\'aticas e Computac\~ao ICMC/USP, S\~ao Carlos, Brazil}
\email{{alvarenga@icmc.usp.br}}

\author{Herivelto Borges}
\address{\rm Herivelto Borges, Instituto de Ci\^encias Matem\'aticas e Computac\~ao ICMC/USP, S\~ao Carlos, Brazil}
\email{hborges@icmc.usp.br}

\begin{document}

\begin{abstract} Let $\Fq$ be the finite field with cardinality $q$, where $q$ is a prime power. Given a finite field extension $\Fqn$ over  $\Fq$ and $a,b \in \Fq^{*}$, we investigate in this article the number $N_n(a,b)$ of elements in $\Fqn$ whose norm equals $a$ and trace equals $b$. Our approach to probe $N_n(a,b)$ is to connect it  with the number of rational points on certain Artin-Schreier curve. After establish an improvement of the Hasse-Weil bound for that Artin-Schreier curve, we improve the known estimates for $N_n(a,b)$ when (roughly speaking) $n \geq \sqrt{q}-1. $
Moreover, we use this approach to improve the bound given by Moisio and Wan in \cite{moisio-wan} for the number of rational points on the toric Calabi-Yau variety studied by Rojas-Leon and Wan in \cite{rojas-wan}. We finish the paper with explicit calculations of $N_n(a,b)$ and an application to the number of irreducible monic polynomials in an arithmetic progression.

\end{abstract} 

\maketitle


\section{Introduction}
\label{introduction}

Let $p$ be a prime number, $q$ be a power of $p$ and $\Fq$ be the finite field with $q$ elements. Given a positive integer $n \geq 2$ and $a,b \in \Fq$, we define 
\[N_n(a,b) := \#\big\{ z \in \mathbb{F}_{q^n} \; |\; \Norm(z) = a \; \text{ and }\; \Tr(z) =b \big\},\]
where $\Norm$ and $\Tr$ are respectively the norm and trace maps from $\Fqn$ to $\Fq.$

The interest in such number grew up after the work of Nicholas Katz in \cite{katz}. Motivated by some estimates of Soto-Andrade sums, Katz proves in the nineties the following bound for $N_n(a,b)$ in the case $a \in \Fq^{*}$
\begin{equation} \label{katzbound}
\Big| N_n(a,b) - \frac{q^{n}-1}{q(q-1)} \Big| \leq n q^{\frac{n-2}{2}}.
\end{equation} 
The Katz proof of \eqref{katzbound} is based on a computation of the Fourier transform of a function associated to  $N_n(a,b)$ in terms of some Kloosterman sheaf and the Lefschetz trace formula. See \cite[Sec. Proof of Theorem $4$]{katz} for further details. 

Driven by several applications, it is of interest to give a sharp bound for the number $N_n(a,b)$.  This article is mainly concerned with provide a better estimate for $N_n(a,b)$
when $ab \neq 0$. As we can see below, Katz's bound \eqref{katzbound} was already uniformly improved in \cite{moisio-wan}. Therefore our task in this work is to improve the bound given in \cite{moisio-wan}. We shall give more details later. Before, let us investigate $N_n(a,b)$ for some natural inputs.


\subsection*{Elementary cases} Let us briefly survey the elementary cases of $N_n(a,b)$ i.e.\ considering either $a=0$ or $b=0$ or $n=2,3.$ 

 We first observe that for $a=0$, $N_n(0,b)$ is trivial for every $n \geq 2$. Namely, 
$$ N_n(0,b) =  \begin{cases} 1 & \hfill \text{ for } b=0,  \\ 
0 & \hfill \text{ for } b \neq 0. \end{cases} $$

For $b=0$, applying Weil bound for characters (Gauss) sums and elementary group theory, Marko Moisio improves in  \cite[Cor. 3.3]{moisio} the Katz bound \eqref{katzbound} 
\begin{equation} \label{moisiobound}
\Big| N_n(a,0) - \frac{q^{n-1}-1}{q-1} \Big| \leq (d-1) q^{\frac{n-2}{2}},
\end{equation}
where $d := \gcd(n,q-1)$. 

When $n=2$, we present the calculation of $N_2(a,b)$ in the generic case, i.e.\ for $ab \neq 0$ and $p>2.$ 
 Note that if $z \in \mathbb{F}_{q^2}$ is such that $\Norm(z)=a$ and $\Tr(z)=b$, then 
$$ z^{q+1} =a \quad \text{and} \quad z^q + z = b.$$ 
Hence, $z$ is a root of $f(T) := T^2 -bT+a \in \Fq[T]$, which implies $N_2(a,b) \leq 2.$ 

Moreover,  
$N_2(a,b) =2$ if and only if $\Delta := b^2 -4a$ is not a square in $\Fq.$ 
One can check that the number of pairs $(a,b) \in \Fq^{*} \times \Fq^{*}$ such that $\Delta$ is not a square in $\Fq$ is $(q-1)^2/2.$ In addition, note that $N_2(a,b)=1$ if and only if $\Delta =0,$
which comprises $q-1$ pairs $(b^2/4,b) \in \Fq^{*} \times \Fq^{*}$. Therefore there exist $(q-1)(q-3)/2$ pairs $(a,b) \in \Fq^{*} \times \Fq^{*}$ for which $N_2(a,b)=0.$ Putting all together yields the following 
\[ N_2(a,b) = \begin{cases} 
0 \hfill & \text{ if } \;b^2-4a \text{ is a square in } \Fq^{*} \\ 
1 \hfill & \text{ if } \; b^2-4a =0 \\   
2  \hfill & \text{ if } \; b^2-4a \text{ is a non-square in } \Fq^{*}, \\
\end{cases}
\]
and the number of pairs $(a,b) \in \Fq^{*} \times \Fq^{*}$ which happens each respective case above are 
$(q-1)(q-3)/2,\; q-1$ and $(q-1)^{2}/2$.


The case $n=3$ is extensively studied by Moisio in \cite{moisio} where the author associates  $N_3(a,b)$ with the number of $\Fq$-rational points on certain cubic curve.  Using some arithmetic manipulations in the case that cubic curve is singular and applying the Hasse theorem for elliptic curves in the non-singular case, Moisio achieves the following sharp bound (for $a \neq 0$) when $n=3$
\[ 3 \Big\lceil \frac{q+1-2\sqrt{q}}{3} \Big\rceil \leq N_3(a,b) \leq 3 \Big\lfloor \frac{q+1+2 \sqrt{q}}{3} \Big\rfloor \]
which improves  Katz's bound \eqref{katzbound}. See \cite[Thm. 5.1 and 5.3]{moisio} for further details.

The seemingly first non-trivial case is for $q=3$.  We shall present in the section \ref{sec-expcomp} the explicit calculation of $N_n(a,b)$  for $q=3, 4, 5$ . 

\subsection*{Some applications} Besides the natural insterest on $N_n(a,b)$, Katz's bound \eqref{katzbound} for this number has been used throught the literature for some applications on the arithmetic of finite fields. For $n=3$ it  plays an important role for the proof of existence of cubic primitive normal polynomials with given norm and trace in \cite{h-cohen}.

Moreover, in \cite{moisio} Katz's bound \eqref{katzbound} for $N_n(a,b)$ is applied to study the number of irreducible monic polynomials of degree $n$ in an arithmetic progression or, equivalently, 
the number $P_n(a,b)$ of irreducible polynomials of the following form 
$$ T^n - a T^{n-1} + \cdots + (-1)^{n}b \in \Fq[T].$$ 
In \cite[Thm. 5.1]{wan} Daqing Wan obtains an effective bound for $P_n(a,b)$ depending on $n$ and $q$. 
Later in \cite{moisio},  Moisio applies the Katz bound \eqref{katzbound} to obtain the improvement on the bound given by Wan when $n< \tfrac{3}{2}(q-1).$  We present in the last section these bounds, some improvements and explicit formulas  for $q=2,3,4,5$.


\subsection*{Improvements on Katz's bound} Besides  the cases $b=0$ and $n=3$ pointed out above, we have the following improvements on Katz's bound \eqref{katzbound}.   

Motivated by the calculation on the number of irreducible polynomials of degree $n$ in $\Fq[T]$ with prescribed norm and trace coefficients, also in \cite{moisio} the author applies the Deligne estimate for hyper-Kloosterman sums to obtain an improvement of Katz's bound \eqref{katzbound} when $n$ is a power of $p$.

A uniform improvement for Katz's bound \eqref{katzbound} is given in \cite{moisio-wan} by Moisio and Wan. Namely, the authors provide the following bound for $N_n(a,b) $ 
\begin{equation} \label{moisio-wanbound}
\Big| N_n(a,b) - \frac{q^{n-1}-1}{q-1} \Big| \leq (n-1) q^{\frac{n-2}{2}}
\end{equation}
which improves the Katz bound for every $a,b \in \Fq^{*}$ and all $n \geq 2$.  The proof of above bound is based on two steps. First, the authors connect  $N_n(a,b) $ with the number of $\Fq$-rational points on the toric Calabi-Yau hypersurface given by
\[ Y_u \; :\; X_1 + \cdots + X_n + \frac{1}{X_1 \cdots X_n} =1\]
where $u=b/a^{n},$ see \cite[Lemma 2.1]{moisio-wan}. The second step is to apply the cohomological calculation in \cite{rojas-wan} for such toric hypersurface in order to estimate its number of $\Fq$-rational points, see \cite[Thm. 2.2]{moisio-wan}. Moreover, an improvement for $N_{\ell}(a,b)$ is obtained for $\ell$ an odd prime number and $ab \in \Fq^{*}$, which generalizes \cite[Thm. 5.3]{moisio}, cf. \cite[Cor. 2.4]{moisio-wan}.

\subsection*{Intention and scope of this article} As we can observe from previous paragraphs, many different methods were used along the literature to study the number $N_n(a,b)$ of elements with given norm and trace in a finite field extension of $\Fq.$ In this article we propose a novel approach. Since the cases either $a=0$ or $b=0$ are well understood, we focus our attention to the generic case where $ab \in \Fq^{*}$.
We shall connect the problem of count the number of elements with prescribed norm and trace with the problem of estimate the number of rational points on a curve defined over a finite field. Namely, using the Hilbert 90 theorem our first main result, cf. Theorem \ref{thmlink}, asserts
\[ N_n(a,b) = \frac{\#X(\Fqn) -1}{q(q-1)} \] 
where $X(\Fqn)$ stands for the set of $\Fqn$-rational points of $X$ the non-singular projective Artin-Schreier curve defined over $\Fqn$ whose affine equation is given by
\[ y^q- y = \alpha x^{q-1} - \beta \]
and $\alpha, \beta \in \Fqn$ are such that $\Norm(\alpha)=a$ and $\Tr(\beta)=b$.

Writing the number $\# X(\Fqn)$ in terms of Gauss sums,  our second main theorem concerns an improvement of the Hasse-Weil bound for that Artin-Schreier curve, see Theorem \ref{improv-HW}. 
Hence, as a corollary of Theorem \ref{improv-HW}, we can improve Moisio-Wan bound \eqref{moisio-wanbound} for, roughly speaking, $n \geq \sqrt{q}-1,$ see Corollary \ref{ASbound2}. 
 Moreover, we present some others special cases of improvement for  Moisio-Wan bound \eqref{moisio-wanbound}, for instance when $q-1$ divides $n$.

We also improve the estimate in \cite[Thm. 2.2]{moisio-wan} for the number of $\Fq$-rational points on above toric Calabi-Yau variety $Y_u$ when, roughly speaking, $n \geq \sqrt{q}-1,$ see Corollary \ref{improvement-toricvariety}. 

Applying our new approach, i.e.\ Theorem \ref{thmlink}, we can calculate precisely $N_n(a,b) $ in the first non-trivial cases, namely for $q=3,4,5$. As an application, we improve in some cases the already known bounds for $P_n(a,b)$ given by Wan in \cite{wan} and Moisio in \cite{moisio}.

\subsection*{Content overview} Section \ref{sec-one} is dedicated to the connection between $N_n(a,b)$ and the number of rational points on certain Artin-Schreier curves. Section \ref{sec-two} is the article's core, it contains an improvement of the Hasse-Weil bound and as corollary, a partially improvement of Moisio-Wan bound \eqref{moisio-wanbound}. In the section \ref{sec-expcomp} we exhibit some explicit computations of $N_n(a,b)$ in the first non-trivial cases. The last section is dedicated to some applications on the number $P_n(a,b)$ of irreducible polynomials with given coefficients.

\subsection*{Acknowledgments} 
   The first author thanks Antonio Rojas-Leon for fruitfull exchange of emails and FAPESP - grant numbers 2017/21259-3 and 2022/09476-7 - for the financial support.


\section{Artin-Schreier curves and \texorpdfstring{$N_n(a,b)$}{N_n(a,b)}} 
\label{sec-one}

Based on the 90th theorem in David Hilbert's famous Zahlbericht \cite{hilbert}\footnote{This is actually the English version of cited book.} (or Hilbert's Theorem 90)\footnote{Actually, this theorem is originally due to Ernst Kummer, cf.\  \cite[Pag. 213]{kummer}.} we might associate $N_n(a,b)$ with the number of $\Fqn$-rational points on certain non-singular Artin-Schreier plane curve. Next, the Hasse-Weil theorem  yields an improvement of Moisio-Wan bound for, roughly speaking, $n > q-1$.  

\begin{thm} \label{thmlink} Let $a,b \in \Fq^{*}$ and $\alpha, \beta \in \Fqn$ be such that $\Norm(\alpha) =a$ and $\Tr(\beta) =b.$ Let $X$ denote the non-singular projective curve defined over $\Fqn$ given by affine equation
\[y^q-y = \alpha x^{q-1} - \beta.\]
If $\#X(\Fqn)$ stands for the number of $\Fqn$-rational points of $X$, then
\[ N_n(a,b) = \frac{\# X(\Fqn) -1 }{q(q-1)}.\]
\end{thm}

\begin{proof}  Let $z \in \Fqn$ be such that $\Norm(z)=a$ and $\Tr(z)=b$. Since $\Norm$ and $\Tr$ are surjective maps, we first observe that for any $a,b \in \Fq^{*}$ indeed there exists $\alpha, \beta \in \Fqn$ be such that $\Norm(\alpha) =a$ and $\Tr(\beta) =b.$ Hence Hilbert's Theorem 90 yields
 \[ z = \alpha x^{q-1} \quad \text{ and } \quad z = y^q - y + \beta \]
for some $x \in \Fqn^{*}$ and $y \in \Fqn.$ Thus, 
\[ y^q- y = \alpha x^{q-1} - \beta\]
for each $z \in \Fqn$ as before. Let $X$ the non-singular projective curve defined over $\Fqn$ whose affine representation is given by above equation.  

Since $(sx)^{q-1} =x^{q-1}$ and $(t+y)^q - (t+y) = y^q + y$ for any $s \in \Fq^{*}$ and $t \in \Fq$, each $z \in \Fqn$ satisfying $\Norm(z)=a$ and $\Tr(z)=b$ gives rise to $q(q-1)$ affine rational points of $X$. Conversely, each such $z$ arises from a set of $q(q-1)$ distinct affine rational points of $X.$ The proof follows by observing that $X$ has only one $\Fqn$-rational point at infinity. 
\end{proof}

\begin{cor} Let $a,b \in \Fq^{*}$ and $n$ denotes an integer bigger than one. Then
\begin{equation} \label{ASbound1}
\big| N_n(a,b) - \frac{q^{n-1}-1}{(q-1)} \big| \leq (q-2) q^{\frac{n-2}{2}} + \frac{1}{q-1},
\end{equation} 
which improves Moisio-Wan bound \eqref{moisio-wanbound} in general when $n>q-1.$ For some particular values we also obtain an improvement  when $n=q-1$.
\end{cor}

\begin{proof} From previous theorem 
\[ q(q-1) N_n(a,b) = \# X(\Fqn) -1  \]
where $X$ is the projective non-singular Artin-Schreier curve given by
\[ y^q - y = \alpha x^{q-1} - \beta\] 
where $\alpha, \beta \in \Fqn$ are such that $\Norm(\alpha)=a$ and $\Tr(\beta)=b.$

From the celebrated Hasse-Weil theorem
\[ \big| \# X(\Fqn) - (q^n+1) \big| \leq (q-1)(q-2) q^{\frac{n}{2}},\]
which implies 
\[  q^n - (q-1)(q-2)  q^{\frac{n}{2}} \leq \# X(\Fqn) -1 \leq q^n + (q-1)(q-2)  q^{\frac{n}{2}}.\]
Replacing  $\# X(\Fqn) -1 $ by $q(q-1)N_n(a,b) $ yields
\[
\frac{q^{n-1}}{q-1} - (q-2)q^{\frac{n-2}{2}} \leq N_n(a,b) \leq \frac{q^{n-1}}{q-1} + (q-2)q^{\frac{n-2}{2}}. \label{eq:star} \tag{$\star$}
\]
Thus, 
\[ N_n(a,b) - \frac{q^{n-1}-1}{q-1} \leq (q-2)q^{\frac{n-2}{2}} + \frac{1}{q-1}\]
as desired. To obtain the left side inequality desired and finish the proof, we observe that the left side inequality in \eqref{eq:star} is slightly stronger than what we stated.  
\end{proof}

 In the next section we shall provide an improvement of Moisio-Wan bound uniformly for $n \geq \sqrt{q}-1$. 



\section{Improvements on Hasse-Weil and Moisio-Wan bounds} \label{sec-two}

In this section we pursue our search for an improvement of Moisio-Wan bound \eqref{moisio-wanbound}.  Translating the number of $\Fqn$-rational points on the Artin-Schreier curve given before in terms of some Gauss sums,  we are able to improve the Hasse-Weil bound.  For a finite abelian group $G$ we denote by $\widehat{G}$ its character group.

\begin{thm} \label{improv-HW} Let $X$ be the non-singular projective Artin-Schreier curve defined over $\Fqn$ whose affine equation is given by
\[y^q - y = \alpha x^{q-1} - \beta \]
where $\alpha, \beta \in \Fqn^{*}$.
Let $X(\Fqn)$ stand for the set of $\Fqn$-rational points of $X$.  Then the following hold: 

\begin{enumerate}[label=\textbf{{\upshape(\roman*)}}]
\item If $\Tr(\beta) \neq 0$,
\[ \big|\#X(\Fqn) - (q^n+1) \big| \leq q^{n/2}(d-1) + q^{(n+1)/2}(q-1-d)\]
which improves Hasse-Weil for all $q$ and $n$.

\item If $\Tr(\beta) =0$
\[\big|\#X(\Fqn) - (q^n+1-q) \big| \leq (q-1)(d-1)q^{n/2}\]
which improves Hasse-Weil for all $q$ and $n$.
\end{enumerate}
where $d:= \gcd(q-1,n).$ 
\end{thm}

\begin{proof}
As we have seen before, $X$ is a smooth projective curve with only one rational point at infinity.  Hence, in the following we suppose   $\#X(\Fqn)$ stands for the number of affine $\Fqn$-rational  points of $X$ and at the end we add the point at infinity. 

Observe that if $(x_0,y_0) \in X(\Fqn)$,  then $(x_0,y_0+c) \in X(\Fqn)$ for all $c \in \Fq$ and thus $q$ divides $X(\Fqn).$ On the other hand,  Hilbert's Theorem 90 assets that if 
$$y^q -q =  \alpha x^{q-1} - \beta,$$  
then $\Tr( \alpha x^{q-1} - \beta)=0$.  Hence
\[ \#X(\Fqn) = q \cdot \#\big\{ x \in \Fqn \; | \; \Tr( \alpha x^{q-1} - \beta)=0 \big\}. \]  
Given $f \in \Fq[T]$,  let $N_ f(0):= \big\{ x \in \Fqn \; |\; \Tr(f(x)) =0 \big\}.$ Thus, 
\[  N_ f(0) = \sum_{\substack{ x \in \Fqn \\ \Tr(f(x))=0)}} 1 = \sum_{\substack{ x \in \Fqn \\ \Tr(f(x))=0}}  \chi(\Tr(f(x))\]
where $\chi$ is a additive character of $\Fq$.  Since, 
\[\sum_{\chi \in \widehat{\Fq}} \chi(a) \chi(b) = \begin{cases} 
0 \hfill \text{ if } a \neq b \\
q \hfill \text{ if } a=b          
\end{cases}\]
for $a,b \in \Fq,$ then 
\[ N_ f(0) = \frac{1}{q} \sum_{ x \in \Fqn} \sum_{\chi \in \widehat{\Fq}} \chi(\Tr(f(x))).\] 
For $t \in \Fq$,  we denote $\chi_t$ the additive character of $\Fq$ defined by $\chi_t(c):= \chi(tc)$ for all $c \in \Fq$.  Recall that as $t$ runs over $\Fq$,  $\chi_t $ runs through all additive character of $\Fq$,  cf.\ \cite[Thm. 5.7]{lidl}. 

Let $\chi$ stand for the canonical additive character of $\Fq$,  from the previous discussion we conclude
\begin{align*}
\#X(\Fqn) & = \sum_{ x \in \Fqn} \sum_{t \in \Fq} \chi(t \Tr(\alpha x^{q-1} - \beta)) \\
&=  q^n + \sum_{ x \in \Fqn} \sum_{t \in \Fq^{*}} \chi(t \Tr(\alpha x^{q-1} - \beta)) \\
& =q^n + \sum_{ x \in \Fqn} \sum_{t \in \Fq^{*}} \chi(t \Tr(\alpha x^{q-1}))\; \overline{\chi}(t \Tr( \beta)) .
\end{align*} 
where $\overline{\chi}$ denotes the inverse of $\chi$ in $\widehat{\Fq}$.  By \cite[Pag. 191]{lidl} 
\[\sum_{ x \in \Fqn}  \chi(t \Tr(\alpha x^{q-1})) = \sum_{ x \in \Fqn}  \mu(t \alpha x^{q-1}) \]
where $\mu$ is the canonical additive character of $\Fqn.$ Given a multiplicative $\psi$ and an additive $\mu$ character we denote by $G(\psi, \mu)$ its Gauss sum.  The \cite[Thm. 5.30]{lidl}  yields
\[\sum_{ x \in \Fqn}  \mu(t \alpha x^{q-1}) = \sum_{j=1}^{q-2} \overline{\psi}^{j}(t \alpha) G(\psi^j, \mu) \]
where $\psi$ is a multiplicative character on $\Fqn^{*}$ of order $q-1.$ Since the subgroup of $\widehat{\Fqn^{*}}$ of order $q-1$ is given by $\big\{ \lambda \circ \Norm \;|\; \lambda \in \widehat{\Fq^{*}}\big\}$ then
\[\sum_{ x \in \Fqn}  \mu(t \alpha x^{q-1}) = 1 + \sum_{\lambda \in \widehat{\Fq^{*}}} \overline{\lambda} (\Norm(t \alpha))\; G(\lambda\circ \Norm,  \mu) . \] 
We denote $a := \Norm(\alpha) \in \Fq^{*}$  and $b := \Tr(\beta) \in \Fq.$ Hence the number of affine $\Fqn$-rational points of $X$ can been written as follows
\begin{align*}
\#X(\Fqn) - q^n & = \sum_{t \in \Fq^{*}}  \Big(  1 + \sum_{\lambda \in \widehat{\Fq^{*}}} \overline{\lambda} (\Norm(t \alpha))\; G(\lambda\circ \Norm,  \mu) \Big) \overline{\chi}(t \Tr( \beta)) \\ 
& =  \sum_{t \in \Fq^{*}}  \Big(  1 + \sum_{\lambda \in \widehat{\Fq^{*}}} \overline{\lambda} (\Norm(t \alpha))\; G(\lambda\circ \Norm,  \mu) \Big) \overline{\chi}(t b) \\
& = -1 + \sum_{t \in \Fq^{*}} \sum_{\lambda \in \widehat{\Fq^{*}}}  G(\lambda\circ \Norm,  \mu) \;  \overline{\lambda} (\Norm(t \alpha))  \overline{\chi}(t b) \\
& = -1 +  \sum_{\lambda \in \widehat{\Fq^{*}}}  G(\lambda\circ \Norm,  \mu) \cdot \overline{\lambda}(\alpha)  \sum_{t \in \Fq^{*}} \overline{\lambda}^n (t)  \overline{\chi}_{b}(t ) \\
& = -1 +  \sum_{\lambda \in \widehat{\Fq^{*}}}  G(\lambda\circ \Norm,  \mu) \cdot G(\overline{\lambda}^n ,  \overline{\chi}_{ab}) \\ 
& = \sum_{\lambda \in \widehat{\Fq^{*}} \setminus \{\lambda_0\}}  G(\lambda\circ \Norm,  \mu) \cdot G(\overline{\lambda}^n ,  \overline{\chi}_{ab})
\end{align*}
where $\lambda_0$ is the trivial multiplicative character of $\Fq^{*}.$

We observe that $\overline{\lambda}^n = \lambda_0$ if and only if the order of $\overline{\lambda}$ divides $d:= \gcd(n, q-1).$ Moreover, since $\widehat{\Fq^{*}}$ is a cyclic group, there exists a unique $\psi \in \widehat{\Fq^{*}} $ character of order $d$ such that the characters in  $\widehat{\Fq^{*}} $ with order dividing $d$ are exactly given by $\psi^j$ for $j=0, \ldots, d-1$. In particular, the number of characters whose order divides $d$ is exactly $d$. In the following we split our analyses in the cases that $b$ is or not zero.  

Suppose $b$ nonzero.  By \cite[Thm. 5.11]{lidl} and previous discussion yields
\begin{align*}
|\#X(\Fqn) - q^n | & \leq \sum_{\lambda \in \widehat{\Fq^{*}} \setminus \{\lambda_0\}}  | G(\lambda\circ \Norm,  \mu)| \cdot |G(\overline{\lambda}^n ,  \overline{\chi}_{ab})|\\
& = \sum_{\lambda \in \widehat{\Fq^{*}} \setminus \{\lambda_0\}} q^{n/2} \; 
\cdot | G(\overline{\lambda}^n ,  \overline{\chi}_{ab})| \\
& = q^{n/2}(d-1) + q^{(n+1)/2}(q-1-d).  
\end{align*}


Suppose $b$ zero.  Let $\chi_0$ be the trivial additive character of $\Fq$, then   
 \begin{align*}
\#X(\Fqn) - q^n  &= -1 + \sum_{\lambda \in \widehat{\Fq^{*}}}  G(\lambda\circ \Norm,  \mu) \cdot G(\overline{\lambda}^n ,  \chi_{0}) \\
&   =  -q + \sum_{\lambda \in \widehat{\Fq^{*}} \setminus \{\lambda_0\}}   
G(\lambda\circ \Norm,  \mu) \cdot G(\overline{\lambda}^n ,  \chi_{0})  \\
& = -q + \sum_{j=1}^{d-1} G(\psi^j \circ \Norm,  \mu) \cdot G(\lambda_0 ,  \chi_{0})\\
& = -q + (q-1) \sum_{j=1}^{d-1} G(\psi^j \circ \Norm,  \mu).
\end{align*}
Therefore by \cite[5.11]{lidl},
\[
|\#X(\Fqn) - q^n +q | \leq (q-1)(d-1)q^{n/2},  \]
which completes the proof. 
 
  \end{proof}
  
See \cite{borges-brochero-oliveira23} for a generalization of above theorem to Artin-Schreier hypersurfaces.

\begin{rem} Artin-Schreier curves plays an important role in the theory of curves of finite fields since it often yield examples of interesting phenomena, see for instance \cite{garcia-stich}. Moreover, theses curves have many applications on code theory and cryptography e.g.\ \cite{geer-vlugt} and \cite{koblitz}. Improvements for the Hasse-Weil bound over Artin-Schreier curves can be found in \cite{rojas-wan-artin}, \cite{coulter} and \cite{rojas-manyaut}. See also \cite{rojas-rationality} for a more general approach when one has a large abelian finite group acting on a variety. Many others authors have been working on improvements of Hasse-Weil bound, although none of them implies an improvement for Moisio-Wan bound \eqref{moisio-wanbound}.  \end{rem}     

\begin{cor}  \label{cor-ASbound2} Let $a,b \in \Fq^{*}$ and $n$ denotes an integer bigger than one.  Then 
\begin{equation} \label{ASbound2}
\Big| N_n(a,b) - \tfrac{q^{n-1}-1}{q-1} \Big| \leq 
\frac{1+ (q-2)q^{(n-1)/2} - q^{(n-2)/2}(q^{1/2}-1)(d-1) }{q-1}.
\end{equation} 
where $d:= \gcd(q-1,n)$. 
Moreover, it improves Moisio-Wan bound \eqref{moisio-wanbound}  when either $q-1 \mid n$ or $q-1 \nmid n$ and $n > \frac{q-2}{q-1}\sqrt{q}-1.$ 
\end{cor}

\begin{proof} From Theorem \ref{thmlink} yields
\[ N_n(a,b) = \frac{\#X(\Fqn)-1}{q(q-1)}\]
where $X$ is the non-singular projective Artin-Schreier curve given in the above theorem.  Since $b \neq 0,$ by the previous theorem 
\[ \big|\#X(\Fqn) - (q^n+1) \big| \leq (q-2)q^{(n+1)/2} - q^{n/2}(q^{1/2}-1)(d-1)\]
where $d:= \gcd(q-1,n)$. Thus
\[ N_n(a,b) - \frac{  q^{n-1} -1}{q-1} \leq \frac{1+ (q-2)q^{(n-1)/2} - q^{(n-2)/2}(q^{1/2}-1)(d-1) }{q-1} \]
and on the other side
\[ N_n(a,b) - \frac{  q^{n-1} -1}{q-1}  \geq \frac{1-(q-2)q^{(n-1)/2} + q^{(n-2)/2}(q^{1/2}-1)(d-1)}{q-1}. \]
We conclude the proof by observing that the above last inequality is slightly stronger than what we stated.     
\end{proof}

Let $u \in \Fq^{*}$ and  $Y_u$ be the toric Calabi-Yau hypersurface
given by 
\[ X_1 + \cdots + X_{n-1} + \frac{1}{X_1 \cdots X_{n-1}} =1.\]
In \cite{rojas-wan} the authors provide a detailed investigation of \'etale  cohomology for such toric hypersurfaces. Applying these detailed cohomological calculations, in \cite{moisio-wan} the authors give the following estimate for the number of rational points on $Y_u$. 

\begin{thm}{\cite[Thm. 2.2]{moisio-wan}} \label{thm-mwtoric} Let $u \in \Fq^{*}$. We denote by $Y_u(\Fq)$ the set of $\Fq$-rational points on $Y_u$ the toric Calabi-Yau hypersurface given as above. Then
\[ \Big| \# Y_u(\Fq) - \frac{(q-1)^{n-1} - (-1)^{n-1}}{q} \Big| \leq (n-1)q^{(n-1)/2}.\]
\end{thm}

The following corollary improves the bound for the number of $\Fq$-rational points on the toric Calabi-Yau variety $Y_u$ used by Moisio-Wan in \cite{moisio-wan} to improve Katz bound \eqref{katzbound}. 


\begin{cor} \label{improvement-toricvariety} Let $u = a/b^n \in \Fq^{*}$ and  $Y_u$ be the toric Calabi-Yau hypersurface
given in the above discussion. 
Let $Y_u(\Fq)$ denote the set of $\Fq$-rational points on $Y_u$, then 
\[ \Big| \# Y_u(\Fq) - \frac{(q-1)^{n-1} - (-1)^{n-1}}{q} \Big| \leq  
\frac{1+ (q-2)q^{(n-1)/2} - q^{(n-2)/2}(q^{1/2}-1)(d-1) }{q-1}
\] 
where $d:= \gcd(q-1,n)$. Moreover, it improves the bound given in the Theorem \ref{thm-mwtoric} when either $q-1 \mid n$ or $q-1 \nmid n$ and $n > \frac{q-2}{q-1}\sqrt{q}-1.$ 

\end{cor}

\begin{proof} The \cite[Lemma 2.1]{moisio-wan} gives a relation between $\#Y_u(\Fq)$ and $N_n(a,b)$
\[ N_n(a,b) = \frac{q^{n-1}-1}{q-1} + (-1)^{n-1} \left(\#Y_u(\Fq) - \frac{(q-1)^{n-1}-(-1)^{n-1}}{q}\right).\]
From Theorem \ref{thmlink} yields
\[ (-1)^{n-1} \left(\#Y_u(\Fq) - \frac{(q-1)^{n-1}-(-1)^{n-1}}{q}\right) = \frac{\#X(\Fqn) -1-q^n + q}{q(q-1)} \]
where $X$ is the non-singular projective Artin-Schreier curve defined over $\Fqn$ given by the following affine equation
\[ y^q - y = \alpha x^{q-1} - \beta\]
where $\Norm(\alpha)=a$ and $\Tr(\beta) = b.$ We conclude the proof applying Theorem \ref{improv-HW}. 
\end{proof}

\begin{cor} Let $X$ be the non-singular projective Artin-Schreier curve defined over $\Fqn$ whose affine equation is given by
\[y^q - y = \alpha x^{q-1} - \beta \]
where $\alpha, \beta \in \Fqn^{*}$ are such that $\Tr(\beta) \in \Fq^{*}.$ 
The following bound for the number $\# X(\Fqn)$ of $\Fqn$-rational points of $X$
\[ \big| \#X(\Fqn) - (q^n + 1) \big| \leq q + q(q-1)(n-1)q^{(n-2)/2}\]
improves that bound given in Theorem \ref{improv-HW} (and then the Hasse-Weil bound) when 
\[ n < \Big\lfloor \frac{(q-2)q^{i/2}}{1-q^{-1}} \Big\rfloor +1 \] 
where $i=0$ if $q-1 \mid n$ and $i=1$ if $q-1 \nmid n.$
\end{cor}

\begin{proof} Let $Y_u$ be the toric variety given in the previous corollary where 
$u = \Norm(\alpha)/\Tr(\beta)^n$. As in the proof of previous corollary 
\[ \#X(\Fqn) - q^{n} - 1  = (-1)^{n-1} q(q-1) \left(\#Y_u(\Fq) - \frac{(q-1)^{n-1}-(-1)^{n-1}}{q}\right) -q. \]
The \cite[Thm. 2.2]{moisio-wan} asserts 
\[ \Big| \# Y_u(\Fq) - \frac{(q-1)^{n-1} - (-1)^{n-1}}{q} \Big| \leq (n-1)q^{(n-2)/2}. \]
Putting all together we have the desired statement. 
\end{proof}

\begin{rem}  We finish this section with other proof for Corollary \ref{ASbound2}.  Although, we would like to point out the importance of former proof since it contains as key idea an improvement of Hasse-Weil bound. Improvements on the Hasse-Weil bound is of great value since Hasse-Weil theorem is a deep result in mathematics and has found wide applications in
mathematics, theoretical computer science, information theory etc.

As above, we invoke  \cite[Lemma 2.1]{moisio-wan} which asserts 
$N_n(a,b)$
\[ N_n(a,b) = \frac{q^n-1}{q(q-1)} + (-1)^{n-1} \left(\#Y_u(\Fq) - \frac{(q-1)^n}{q(q-1)}\right),\]
where $Y_u$ is the toric Calabi-Yau variety asserted as Corollary \ref{improvement-toricvariety}. Moreover, we can write 
\[ q(q-1)\#Y_u(\Fq) = (q-1)^{n} + \sum_{\lambda \in \widehat{\Fq^{*}}} G(\chi, \lambda)^{n} G(\chi, \overline{\lambda}^n) \overline{\lambda}((-1)^{n}u) \]   
where $\chi$ is the canonical additive character of $\Fq$ and $G(\cdot, \cdot)$ the Gauss sum. 
By \cite[Thm. 5.11]{lidl}, 
\begin{eqnarray*}
\big|q(q-1)\#Y_u(\Fq) - (q-1)^{n} \big| &\leq \sum_{\lambda \in \widehat{\Fq^{*}}} |G(\chi, \lambda)|^{n} \cdot |G(\chi, \overline{\lambda}^n)| \cdot |\overline{\lambda}((-1)^{n}u)| \\
& \leq 1 + (q-2)q^{n/2} \cdot \begin{cases} q^{1/2} \quad \hfill \text{ if } q-1 \nmid n \\
1 \quad \hfill \text{ if } q-1 \mid n
\end{cases}. 
\end{eqnarray*}
Therefore 
\[ \big| N_n(a,b) - \tfrac{q^{n-1}-1}{q-1} \big| \leq \frac{1}{q-1} + \frac{q-2}{q-1} q^{n-i/2} \]
where $i=1$ if $q-1 \nmid n$ and $i=2$ if $q-1 \mid n.$
 
\end{rem}

\section{Explicit computation of \texorpdfstring{$N_n(a,b)$}{N_n(a,b)} }
\label{sec-expcomp}

We end this article proposing the interesting problem of finding the actual value of $N_n(a,b).$ Namely, for arbitrarily large $n \geq 3$, compute the number $N_n(a,b) $ for small values of $q \geq 2$. 
In the following, we give an answer to this problem
when $q=2,3,4,5$. 


\begin{prop} \label{prop-q=2} Let $n \geq 1$ and 
\[N_n(1,1) := \#\big\{ z \in \mathbb{F}_{2^n} \;\big| \; \Norm(z)=1 \text{ and } \Tr(z)=1 \big\},\] 
then $N_n(1,1) = 2^{n-1}.$
\end{prop} 

\begin{proof} It follows immediately from inequalities \eqref{eq:star} in the proof of Corollary \ref{ASbound1}. 
\end{proof}


The case $q=3$ , with $n$ being a variable, is apparently the first nontrivial case.

\begin{thm} For $n \geq 1$, let $X$ be the non-singular projective Artin-Schreier curve defined over $\mathbb{F}_{3^n}$ whose affine equation is given by
\[y^q - y = \alpha x^{q-1} - \beta \]
where $\alpha, \beta \in \mathbb{F}_{3^n}$ with $\alpha  \neq 0$.
Let $X(\mathbb{F}_{3^n})$ stand for the set of $\mathbb{F}_{3^n}$-rational points of $X$,   $a := \Norm(\alpha)$ and $b := \Tr(\beta)$. Then
the following hold: 
\begin{enumerate}[label=\textbf{{\upshape(\roman*)}}]

\item If $n = 2m$ is even, 
\[ \# X(\mathbb{F}_{3^n}) = \begin{cases} 3^n -2 (-3)^m + 1  \hfill & \text{ if } b =0, \\[4pt]
  3^n + (-3)^{m} +1  \hfill & \text{ if } b \neq 0;
\end{cases}\]

\item If $n$ is odd, 
\[ \# X(\mathbb{F}_{3^n}) = \begin{cases} 3^n + 1  \hfill & \text{ if } b =0, \\[4pt]
  3^n - (-3)^{(n+1)/2} \eta(a) +1  \hfill & \text{ if } b =1, \\[4pt]
 3^n + (-3)^{(n+1)/2} \eta(a) +1  \hfill & \text{ if } b =-1; \\
\end{cases}\]

\end{enumerate}
where $\eta $ is the quadratic character on $\mathbb{F}_{3}$, i.e.\ $\eta(a)=1$ if $a=1$ and $\eta(a)=-1$ if $a=2.$
\end{thm}

\begin{proof} As in the proof of Theorem \ref{improv-HW} we suppose $\# X(\mathbb{F}_{3^n}) $ stands for the number of affine $\mathbb{F}_{3^n}$-rational points. Then 
\[\# X(\mathbb{F}_{3^n}) = 3^n + \sum_{\lambda \in \widehat{\F_{3}^{*}}\setminus \{\lambda_0\}} G(\lambda \circ \Norm, \mu ) \cdot G(\overline{\lambda}^n, \overline{\chi}_{ab})\] 
where $a := \Norm(\alpha)$, $b := \Tr(\beta)$, $\mu$ (resp. $\chi$) is the canonical additive character of $\F_{3^n}$ (resp. $\F_3$), $\lambda_0$ is the trivial multiplicative character in $\widehat{\F_{3}^{*}}$ and $G(\cdot, \cdot)$ denotes the Gauss sum. 

We first observe that the above sum takes indeed only $\eta \in \widehat{\F_{3}^{*}}$ the quadratic character. Moreover, since 
\[ \eta^n = \begin{cases} \eta \hfill &\text{ if $n$ is odd, and } \\
\lambda_0 \hfill &\text{ if $n$ is even} 
\end{cases}\]
and the Davenport-Hasse theorem (see \cite[Thm. 5.14]{lidl}) yields
\begin{equation} \label{eq9-12}
 X(\mathbb{F}_{3^n}) = 3^n +  (-1)^{n-1} G(\eta, \chi )^n \cdot \begin{cases}
G(\eta, \overline{\chi}_{ab}) \hfill &\text{ if $n$ is odd} \\
G(\lambda_0, \overline{\chi}_{ab}) &\hfill \text{ if $n$ is even}.  
\end{cases}
\end{equation} 

Suppose  $n$ is even. From \eqref{eq9-12} and \cite[Thm. 5.15]{lidl}
\[ X(\mathbb{F}_{3^n}) = 3^n - (-3)^{n/2} G(\lambda_0, \chi_{ab}). \]
From \cite[Thm. 5.11]{lidl} follows that
\[ G(\lambda_0, \chi_{ab}) = \begin{cases} 2 \hfill & \text{ if $b=0$} \\
-1 \hfill & \text{ if $b \neq 0$,} 
\end{cases} \]
and we have the desired.

Suppose now $n$ is odd. By \cite[Thm 5.15, Thm. 5.12]{lidl} (in this order) 
\begin{align*}
\# X(\mathbb{F}_{3^n}) & = 3^n + (-3)^{n/2} G(\eta, \overline{\chi}_{ab}) \\
& = 3^n - (-3)^{n/2} \eta(a) G(\eta, \chi_{b}).
\end{align*}
By \cite[Thm. 5.11, Thm. 5.12, Thm. 5.15]{lidl} yields
\[ G(\eta, \chi_b) = \begin{cases} 0 \hfill & \text{ if $b=0$} \\
(-3)^{1/2} \hfill & \text{ if $b=1$} \\
- (-3)^{1/2} \hfill & \text{ if $b=-1$,} 
\end{cases} \]
 which completes the proof. 
\end{proof}

\begin{cor} \label{cor-q=3} Let $a,b \in \F_3^{*}$, $n \geq 1$ and 
\[N_n(a,b) := \#\big\{ z \in \mathbb{F}_{3^n} \;\big| \; \Norm(z)=a \text{ and } \Tr(z)=b \big\}.\] 
Then the following hold: 
\begin{enumerate}[label=\textbf{{\upshape(\roman*)}}]

\item If $n = 2m$ is even, 
\[ N_n(a,b) = \frac{ 3^n + (-3)^{m} }{6};  \]

\item If $n$ is odd, 
\[ N_n(a,b) = \begin{cases} 
  \frac{ 3^n + (-3)^{(n+1)/2} \eta(a)}{6}  \hfill & \text{ if } b =1 \\[6pt]
\frac{ 3^n - (-3)^{(n+1)/2} \eta(a)}{6}  \hfill & \text{ if } b =-1,   \\
\end{cases}\]

\end{enumerate} 
where $\eta $ is the quadratic character on $\mathbb{F}_{3}$
\end{cor}

\begin{proof} This follows from previous theorem and Theorem \ref{thmlink}. 
\end{proof}


\begin{thm} For $n \geq 1$, let $X$ be the non-singular projective Artin-Schreier curve defined over $\mathbb{F}_{4^n}$ whose affine equation is given by
\[y^q - y = \alpha x^{q-1} - \beta \]
where $\alpha, \beta \in \mathbb{F}_{4^n}$ with $\alpha  \neq 0$.
Let $X(\mathbb{F}_{4^n})$ stand for the set of $\mathbb{F}_{4^n}$-rational points of $X$,  
$a := \Norm(\alpha)$ and $b := \Tr(\beta)$. Then the following hold: 
\begin{enumerate}[label=\textbf{{\upshape(\roman*)}}]

\item If $3|n$, 
\[ \# X(\mathbb{F}_{4^n}) = \begin{cases} 4^n + 3 (-1)^{n-1}2^{n+1} + 1  \hfill & \text{ if } b =0, \\[4pt]
 4^n + (-1)^n 2^{n+1} +1  \hfill & \text{ if } b \neq 0; \\
\end{cases}\]

\item If $3 \nmid n$, 
\[ \# X(\mathbb{F}_{4^n}) = \begin{cases} 4^n + 1  \hfill & \text{ if } b =0, 
\\[4pt]
  4^n + (-1)^{n-1} 2^{n+2} +1  \hfill & \text{ if either $a=b=1$ or 
  $\{a, b\} = \{w,w^2\}$,} \\[4pt]
   4^n + (-1)^{n} 2^{n+1} +1  \hfill & \text{ otherwise};
\end{cases}\]
\end{enumerate}
where $w \in \mathbb{F}_{4}$ is a primitive element such that $\mathbb{F}_{4} = \mathbb{F}_2(w)$.  
\end{thm}

\begin{proof} We first observe that $\widehat{\F_{4}^{*}} = \{\lambda_0, \lambda_1, \lambda_2\}$ where 
$$\lambda_j(w^k) := e^{\frac{2\pi i jk}{3}}$$
 for $j=0,1,2$ and $k=0,1,2$. Moreover, $\overline{\lambda_1} = \lambda_2$ and $\widehat{\F_{4}^{*}} = \langle \lambda_1 \rangle.$

As in the proof of Theorem \ref{improv-HW}, we suppose $\# X(\mathbb{F}_{4^n}) $ stands for the number of affine $\mathbb{F}_{4^n}$-rational points. Follows from above discussion and Davenport-Hasse theorem that 
\begin{align*}
 \# X(\mathbb{F}_{4^n}) & = 4^n + \sum_{\lambda \in \widehat{\F_{4}^{*}}\setminus \{\lambda_0\}} G(\lambda \circ \Norm, \mu ) \cdot G(\overline{\lambda}^n, \overline{\chi}_{ab}) \\
 & = 4^n + (-1)^{n-1} \sum_{j=1,2} G(\lambda_j, \chi)^n \cdot G(\overline{\lambda}_j^n, \overline{\chi}_{ab})
 \end{align*}
where $a := \Norm(\alpha)$, $b := \Tr(\beta)$, $\mu$ (resp. $\chi$) is the canonical additive character of $\F_{4^n}$ (resp. $\F_4$),  and $G(\cdot, \cdot)$ denotes the Gauss sum. 

A straightforward computation shows that $G(\lambda_j, \chi) =2$ for $j=1,2$. Moreover, since $\ord(\lambda_j)=3$, by  \cite[Rm. 5.13]{lidl} $\lambda_j(-1)=1$ for $j=1,2.$ Thus
\[ \#X(\mathbb{F}_{4^n}) = 4^n + (-1)^{n-1} 2^n \sum_{j=1,2} G(\lambda_j^n, \chi_{ab}).  \] 

Suppose $3 \mid n.$ In this case $\lambda_j^n = \lambda_0$ for $j=0,1,2$. If $b=0$ by \cite[Thm. 5.11]{lidl} $G(\lambda_0, \chi_0)=3$, thus
\[ \#X(\mathbb{F}_{4^n}) = 4^n +  (-1)^{n-1} 2^n \cdot 6.    \]
If $b \neq 0$, by \cite[Thm. 5.11]{lidl} $G(\lambda_0, \chi_{ab})= -1$, thus
\[ \#X(\mathbb{F}_{4^n}) = 4^n +  (-1)^{n-1} 2^n \cdot (-2).    \]

Suppose now $3 \nmid n.$ Hence
\[ \#X(\mathbb{F}_{4^n}) = 4^n + (-1)^{n-1} 2^n \sum_{j=1,2} G(\lambda_j, \chi_{ab}),  \]
since $\lambda_j^n = \lambda_j$ if $n \equiv 1(\rm{mod} \;3)$ and $\lambda_j^n = \lambda_i$ if $n \equiv 2(\rm{mod} \;3)$ for $\{i,j\} =\{1,2\}.$ If $b=0$, again by \cite[Thm. 5.11]{lidl} $G(\lambda_j, \chi_0) =0$ for $j=1,2$, whence we conclude
$ \#X(\mathbb{F}_{4^n}) = 4^n.$ If $b \neq 0$ yields
\[ \#X(\mathbb{F}_{4^n}) = 4^n + (-1)^{n-1} 2^{n+1} \big( \lambda_2(a) \lambda_2(b) + \lambda_1(a) \lambda_1(b) \big).  \] 
Let $\Lambda(a,b) := \lambda_2(a) \lambda_2(b) + \lambda_1(a) \lambda_1(b)$. A straightforward computation shows that 
\[ \Lambda(a,b) = \begin{cases} 2  \hfill & \text{ if either $a=b=1$ or $\{a,b\} = \{w, w^2\}$} \\
-1 \hfill & \text{ otherwise,} 
\end{cases}\]
which concludes the proof. 
\end{proof}

\begin{cor} \label{cor-q=4} Let $a,b \in \F_4^{*}$, $n \geq 1$ and 
\[N_n(a,b) := \#\big\{ z \in \mathbb{F}_{4^n} \;\big| \; \Norm(z)=a \text{ and } \Tr(z)=b \big\}.\] 
Then the following hold: 
\begin{enumerate}[label=\textbf{{\upshape(\roman*)}}]
\item If $3|n$, 
\[ N_n(a,b) =  \frac{4^n + (-1)^n 2^{n+1}}{12}; \]

\item If $3 \nmid n$, 
\[ N_n(a,b)= \begin{cases} 
  \frac{4^n + (-1)^{n-1} 2^{n+2}}{12}  \hfill & \text{ if either $a=b=1$ or 
  $\{a, b\} = \{w,w^2\}$,} \\[6pt]
   \frac{4^n + (-1)^{n} 2^{n+1}}{12}  \hfill & \text{ otherwise};
\end{cases}\]
\end{enumerate}
where $w \in \mathbb{F}_{4}$ is a primitive element such that $\mathbb{F}_{4} = \mathbb{F}_2(w)$.  
\end{cor}

\begin{proof} This follows from previous theorem and Theorem \ref{thmlink}. 
\end{proof}


\begin{thm} For $n \geq 1$, let $X$ be the non-singular projective Artin-Schreier curve defined over $\mathbb{F}_{5^n}$ whose affine equation is given by
\[y^q - y = \alpha x^{q-1} - \beta \]
where $\alpha, \beta \in \mathbb{F}_{5^n}$ with $\alpha  \neq 0$.
Let $X(\mathbb{F}_{5^n})$ stand for the set of $\mathbb{F}_{5^n}$-rational points of $X$,  
$a := \Norm(\alpha)$ and $b := \Tr(\beta)$. Then the following hold: 
\begin{enumerate}[label=\textbf{{\upshape(\roman*)}}]

\item If $n = 4m$, 
\[ \# X(\mathbb{F}_{5^n}) = \begin{cases} 5^n - 4 \cdot 5^m \big(5^{m} - 2 \cdot\Re(-3+4i)^m \big) + 1  \hfill & \text{ if } b =0, \\[4pt]
 5^n + 5^m \big( 5^{m} - 2 \cdot\Re(-3+4i)^m \big) +1  \hfill & \text{ if } b \neq 0; \\
\end{cases}\]

\item If $n=4m+1$, 
\[ \# X(\mathbb{F}_{5^n}) = \begin{cases} 

5^n + 1  \hfill & \text{ if } b =0,  \\[4pt]

5^n + 5^{m+1} \cdot 2\;\Re(-3+4i)^m + 5^{(n+1)/2} +1 
\hfill & \text{ if $a=b=1,4$ or $\{a, b\} = \{2,3\}$,} \\[4pt]

5^n - 5^{m+1} \cdot 2\; \Re(-3+4i)^m + 5^{(n+1)/2} + 1 
\hfill & \text{ if $a=b=2, 3$ or $\{a, b\} = \{1,4\}$,} \\[4pt]

5^n - 5^{m+1} \cdot 2\; \Im(-3+4i)^m - 5^{(n+1)/2} +1 
\hfill & \text{ if $\{a, b\} = \{1,2\}$ or $ \{3,4\}$,} \\[4pt]

5^n + 5^{m+1} \cdot 2 \;\Im(-3+4i)^m - 5^{(n+1)/2} +1 
 \hfill & \text{ if $\{a, b\} = \{1,3\}$ or $\{2,4\}$};
\end{cases}\]

\item If $n=4m+2$, 
\[ \# X(\mathbb{F}_{5^n}) = \begin{cases} 

5^n - 4 \cdot 5^{n/2} + 1  \hfill & \text{ if } b =0,  \\[4pt]

5^n + 5^{m+1} \cdot 2\;\Re(-1-2i)(-3+4i)^m + 5^{n/2} +1 
\hfill & \text{ if $b \neq 0 $ and $a= \pm b$},\\[4pt]

5^n - 5^{m+1} \cdot 2\;\Re(-1-2i)(-3+4i)^m + 5^{n/2} +1 
\hfill & \text{ if $b \neq 0 $ and $a\neq \pm b$};\\[4pt]
\end{cases}\]

\item If $n=4m+3$, 
\[ \# X(\mathbb{F}_{5^n}) = \begin{cases} 

5^n + 1  \hfill & \text{ if } b =0,  \\[4pt]

5^n - 5^{m+1} \cdot 2\;\Re(-3+4i)^{m+1} + 5^{(n+1)/2} +1 
\hfill & \text{ if $a=b=1,4$ or $\{a, b\} = \{2,3\}$,} \\[4pt]

5^n + 5^{m+1} \cdot 2\; \Re(-3+4i)^{m+1} + 5^{(n+1)/2} + 1 
\hfill & \text{ if $a=b=2, 3$ or $\{a, b\} = \{1,4\}$,} \\[4pt]

5^n + 5^{m+1} \cdot 2\; \Im(-3+4i)^{m+1} - 5^{(n+1)/2} +1 
\hfill & \text{ if $\{a, b\} = \{1,2\}$ or $ \{3,4\}$,} \\[4pt]

5^n - 5^{m+1} \cdot 2 \;\Im(-3+4i)^{m+1} - 5^{(n+1)/2} +1 
 \hfill & \text{ if $\{a, b\} = \{1,3\}$ or $\{2,4\}$};
\end{cases}\]

\end{enumerate}
where $\Re$ (resp. $\Im$) stands for the real (resp. imaginary) part of a complex number. 
\end{thm}

\begin{proof}  We first observe that $\widehat{\F_{5}^{*}} = \{\lambda_0, \lambda_1, \lambda_2, \lambda_3\}$ where 
$$\lambda_j(2^k) := e^{\frac{2\pi i jk}{4}}$$
 for $j=0,1,2,3$ and $k=0,1,2,4$. Moreover, $\overline{\lambda_1} = \lambda_3$, $\lambda_2$ is the quadratic character and $\widehat{\F_{5}^{*}} = \langle \lambda_1 \rangle.$

As in the proof of Theorem \ref{improv-HW}, we suppose $\# X(\mathbb{F}_{5^n}) $ stands for the number of affine $\mathbb{F}_{5^n}$-rational points. 
Since $\lambda_1(-1)=\lambda_3(-1)=-1$ and $\lambda_2(-1)=1$, follows from above discussion, Davenport-Hasse theorem and \cite[Thm. 5.12]{lidl} that 
\begin{multline*}
 \# X(\mathbb{F}_{5^n})  = 5^n + 
 (-1)^{n}  G(\lambda_1, \chi)^n \cdot G(\lambda_3^n,\chi_{ab}) \\ 
 +   (-1)^{n-1}  G(\lambda_2, \chi)^n \cdot G(\lambda_2^n, \chi_{ab})  
  +  (-1)^{n}  G(\lambda_3, \chi)^n \cdot  G(\lambda_1^n, \chi_{ab})
\end{multline*}
where $a := \Norm(\alpha)$, $b := \Tr(\beta)$, $\chi$ is the canonical additive character of $\F_{5}$  and $G(\cdot, \cdot)$ denotes the Gauss sum. 

Let 
\[ \zeta_5 := \frac{\sqrt{5}-1}{4} + i \frac{\sqrt{10+2 \sqrt{5}}}{4}\]
be a primitive fifty root of unity. Hence $\chi(k) = \zeta_5^k$ for every $k \in \F_5.$
A straightforward computation shows that 
\[G(\lambda_1, \chi) =  - \frac{\sqrt{10 - 2 \sqrt{5}}}{2} 
+ i \frac{\sqrt{10 + 2 \sqrt{5}}}{2}, \]
and thus \cite[Thm. 5.12]{lidl} yields 
 \[G(\lambda_3, \chi) =   \frac{\sqrt{10 - 2 \sqrt{5}}}{2} 
+ i \frac{\sqrt{10 + 2 \sqrt{5}}}{2}. \]
Since $\lambda_2$ is the quadratic character of $\F_5^{*}$, by \cite[Thm. 5.15]{lidl},
$ G(\lambda_2, \chi) =   \sqrt{5}. $
Hence with the information that we have so far, 
\begin{multline*}
 \# X(\mathbb{F}_{5^n})  = 5^n + 
(-1)^{n}\left(-\tfrac{\sqrt{10-2 \sqrt{5}}}{2} + i \tfrac{\sqrt{10 + 2 \sqrt{5}}}{2}\right)^n 
\cdot G(\lambda_3^n,\chi_{ab}) \\
+ (-1)^{n-1}  5^{n/2} \cdot G(\lambda_2^n, \chi_{ab})  
+ (-1)^{n}\left(\tfrac{\sqrt{10 - 2 \sqrt{5}}}{2} + i \tfrac{\sqrt{10 + 2 \sqrt{5}}}{2}\right)^n \cdot  G(\lambda_1^n, \chi_{ab}).
\end{multline*}

Suppose $n=4m.$ In this case $\lambda_j^n = \lambda_0$ for $j=0,1,2,3$, $(-1)^{n-1}=-1$ and $(-1)^n =1.$ If $b=0$, follows from previous discussion and \cite[Thm. 5.11]{lidl} that
\begin{align*}
 \# X(\mathbb{F}_{5^n}) & = 5^n - 4  \left( 5^{n/2} 
 - \left(-\tfrac{\sqrt{10-2 \sqrt{5}}}{2} + i \tfrac{\sqrt{10 + 2 \sqrt{5}}}{2}\right)^n 
 - \left(\tfrac{\sqrt{10 - 2 \sqrt{5}}}{2} + i \tfrac{\sqrt{10 + 2 \sqrt{5}}}{2}\right)^n
 \right) \\
 & = 5^n - 4 \left( 5^{2m} - 5^m \cdot 2 \;\Re(-3+i4)^m \right)
\end{align*}
where $\Re$ stands for the real part of a complex number. In the case that $b \neq 0,$ we apply \cite[Thm. 5.11]{lidl} to obtain 
\begin{align*}
 \# X(\mathbb{F}_{5^n}) & = 5^n   +  5^{n/2} 
 - \left(-\tfrac{\sqrt{10-2 \sqrt{5}}}{2} + i \tfrac{\sqrt{10 + 2 \sqrt{5}}}{2}\right)^n 
 - \left(\tfrac{\sqrt{10 - 2 \sqrt{5}}}{2} + i \tfrac{\sqrt{10 + 2 \sqrt{5}}}{2}\right)^n
 \\
 & = 5^n + 5^{2m} - 5^m \cdot 2 \;\Re(-3+i4)^m,
\end{align*}
which completes the proof of $\textbf{\rm{(i)}}$. 

Suppose $n=4m+1.$ In this case $\lambda_j^n = \lambda_j$ for $j=1,2,3$, $(-1)^n = -1$ and $(-1)^{n-1} = 1.$ Thus
\begin{multline*}
 \# X(\mathbb{F}_{5^n})  = 5^n  
-\left(-\tfrac{\sqrt{10-2 \sqrt{5}}}{2} + i \tfrac{\sqrt{10 + 2 \sqrt{5}}}{2}\right)^n 
\cdot G(\lambda_3,\chi_{ab}) \\
+  5^{n/2} \cdot G(\lambda_2, \chi_{ab})  
- \left(\tfrac{\sqrt{10 - 2 \sqrt{5}}}{2} + i \tfrac{\sqrt{10 + 2 \sqrt{5}}}{2}\right)^n \cdot  G(\lambda_1, \chi_{ab}).
\end{multline*}
If $b=0$, follows from above equation and  \cite[Thm. 5.11]{lidl} that 
$ \# X(\mathbb{F}_{5^n})  = 5^n.$  
If $b \neq 0$, follows from above equation and \cite[Thm. 5.12]{lidl} that
\begin{align*}
 \# X(\mathbb{F}_{5^n})  & = 5^n   - \Lambda(a,b) + 5^{(n+1)/2} \lambda_2(ab)
 \\
 & =  5^n + 5^{m+1} \Big( \lambda_1(ab) (-3+4i)^m + \lambda_3(ab) (-3-4i)^m \Big) 
 + 5^{(n+1)/2} \lambda_2(ab).
\end{align*}
where $\Lambda(a,b)$ stands for 
\[ \scalebox{.88}{ $ G(\lambda_1, \chi) G(\lambda_3, \chi)
 \left( \lambda_1(ab) \left(-\tfrac{\sqrt{10-2 \sqrt{5}}}{2} + i \tfrac{\sqrt{10 + 2 \sqrt{5}}}{2}\right)^{4m} 
  + \lambda_3(ab) \left(\tfrac{\sqrt{10 - 2 \sqrt{5}}}{2} + i \tfrac{\sqrt{10 + 2 \sqrt{5}}}{2}\right)^{4m} \right). $ }\]
Applying all the possible values for $a,b \in \F_5^{*}$ in the previous identity for 
$\# X(\mathbb{F}_{5^n})$ we have the statement of item $\textbf{\rm{(ii)}}$.

Suppose $n=4m+2.$  In this case $\lambda_j^n = \lambda_2$ for $j=1,3$, $\lambda_2^n = \lambda_0$,  $(-1)^n = 1$ and $(-1)^{n-1} = - 1.$ Thus
\begin{multline*}
 \# X(\mathbb{F}_{5^n})  = 5^n  
+ \left(-\tfrac{\sqrt{10-2 \sqrt{5}}}{2} + i \tfrac{\sqrt{10 + 2 \sqrt{5}}}{2}\right)^n 
\cdot G(\lambda_2,\chi_{ab}) \\ 
+ \left(\tfrac{\sqrt{10 - 2 \sqrt{5}}}{2} + i \tfrac{\sqrt{10 + 2 \sqrt{5}}}{2}\right)^n \cdot  G(\lambda_2, \chi_{ab}) 
-  5^{n/2} \cdot G(\lambda_0, \chi_{ab}) .
\end{multline*}
If $b=0$, by \cite[Thm. 5.11]{lidl} yields $ \# X(\mathbb{F}_{5^n})  = 5^n - 4 \cdot 5^{n/2}.$ If $b \neq 0$, by \cite[Thm. 5.11, Thm. 5.15]{lidl} yields
\begin{align*}
 \# X(\mathbb{F}_{5^n})  & = 5^n  +  5^{n/2}
+ \sqrt{5} \left( \left(-\tfrac{\sqrt{10-2 \sqrt{5}}}{2} + i \tfrac{\sqrt{10 + 2 \sqrt{5}}}{2}\right)^n 
+  \left(\tfrac{\sqrt{10 - 2 \sqrt{5}}}{2} + i \tfrac{\sqrt{10 + 2 \sqrt{5}}}{2}\right)^n \right)   \\
& = 5^n + 5^{n/2} + 5^{m+1} \lambda_2(ab)\Big( (-3+4i)^m (-1-2i) + (-3-4i)^m(-1-2i) \Big).
\end{align*}
We conclude item $\textbf{\rm{(iii)}}$ by replacing $a,b$ above by all its possible values in $\F_5^{*}.$ 

Suppose $n=4m+3$. In this case $\lambda_1^n = \lambda_3$, $\lambda_2^n = \lambda_2$,
$\lambda_3^n = \lambda_1$,  $(-1)^n = -1$ and $(-1)^{n-1} = 1.$ Thus
\begin{multline*}
 \# X(\mathbb{F}_{5^n})  = 5^n  
- \left(-\tfrac{\sqrt{10-2 \sqrt{5}}}{2} + i \tfrac{\sqrt{10 + 2 \sqrt{5}}}{2}\right)^n 
\cdot G(\lambda_1,\chi_{ab}) \\ 
- \left(\tfrac{\sqrt{10 - 2 \sqrt{5}}}{2} + i \tfrac{\sqrt{10 + 2 \sqrt{5}}}{2}\right)^n \cdot  G(\lambda_3, \chi_{ab}) 
-  5^{n/2} \cdot G(\lambda_2, \chi_{ab}) .
\end{multline*}
If $b=0$, by \cite[Thm. 5.11]{lidl} yields $ \# X(\mathbb{F}_{5^n})  = 5^n.$ 
If $b \neq 0$, by \cite[Thm. 5.12, Thm. 5.15]{lidl} 
\begin{align*}
 \# X(\mathbb{F}_{5^n})  & = 5^n  +  5^{(n+1)/2} \lambda_2(ab) 
- \lambda_3(ab) G(\lambda_1, \chi)^{n+1} - \lambda_1(ab) G(\lambda_3, \chi)^{n+1}   \\
& = 5^n + 5^{(n+1)/2} \lambda_2(ab) - 5^{m+1} 
\Big( \lambda_3(ab) (-3+4i)^{m+1} + \lambda_1(ab) (-3-4i)^{m+1} \Big).
\end{align*}
Valuating $\lambda_j(ab)$ for all possible values of $a,b \in \F_5^{*}$ and $j=1,2,3$,
we conclude item $\textbf{\rm{(iv)}}$ and hence the theorem's proof. 
\end{proof}

\begin{cor} \label{cor-q=5} Let $a,b \in \F_5^{*}$, $n \geq 1$ and 
\[N_n(a,b) := \#\big\{ z \in \mathbb{F}_{5^n} \;\big| \; \Norm(z)=a \text{ and } \Tr(z)=b \big\}.\] 
Then the following hold: 
\begin{enumerate}[label=\textbf{{\upshape(\roman*)}}]

\item If $n = 4m$, 
\[ N_n(a,b)  = \frac{ 5^n + 5^m \big( 5^{m} - 2 \cdot\Re(-3+4i)^m \big)}{20};\]

\item If $n=4m+1$, 
\[ N_n(a,b)  = \begin{cases}

\frac{ 5^n + 5^{m+1} \cdot 2\;\Re(-3+4i)^m + 5^{(n+1)/2}}{20} 
\hfill & \text{ if $a=b=1,4$ or $\{a, b\} = \{2,3\}$,} \\[6pt]

\frac{  5^n - 5^{m+1} \cdot 2\; \Re(-3+4i)^m + 5^{(n+1)/2}}{20}
\hfill & \text{ if $a=b=2, 3$ or $\{a, b\} = \{1,4\}$,} \\[6pt]

\frac{ 5^n - 5^{m+1} \cdot 2\; \Im(-3+4i)^m - 5^{(n+1)/2}}{20}
\hfill & \text{ if$\{a, b\} = \{1,2\}$ or $ \{3,4\}$,} \\[6pt]

\frac{ 5^n + 5^{m+1} \cdot 2 \;\Im(-3+4i)^m - 5^{(n+1)/2}}{20} 
 \hfill & \text{ if$\{a, b\} = \{1,3\}$ or $\{2,4\}$};
\end{cases}\]

\item If $n=4m+2$, 
\[ N_n(a,b)  = \begin{cases} 

\frac{ 5^n + 5^{m+1} \cdot 2\;\Re(-1-2i)(-3+4i)^m + 5^{n/2} }{20}
\hfill & \text{ if $a= \pm b$},\\[6pt]

\frac{ 5^n - 5^{m+1} \cdot 2\;\Re(-1-2i)(-3+4i)^m + 5^{n/2} }{20}
\hfill & \text{ if $a\neq \pm b$};\\[6pt]
\end{cases}\]

\item If $n=4m+3$, 
\[ N_n(a,b) = \begin{cases}

\frac{ 5^n - 5^{m+1} \cdot 2\;\Re(-3+4i)^{m+1} + 5^{(n+1)/2} }{20}
\hfill & \text{ if $a=b=1,4$ or $\{a, b\} = \{2,3\}$,} \\[6pt]

\frac{  5^n + 5^{m+1} \cdot 2\; \Re(-3+4i)^{m+1} + 5^{(n+1)/2} }{20}
\hfill & \text{ if $a=b=2, 3$ or $\{a, b\} = \{1,4\}$,} \\[6pt]

\frac{ 5^n + 5^{m+1} \cdot 2\; \Im(-3+4i)^{m+1} - 5^{(n+1)/2} }{20}
\hfill & \text{ if$\{a, b\} = \{1,2\}$ or $ \{3,4\}$,} \\[6pt]

\frac{ 5^n - 5^{m+1} \cdot 2 \;\Im(-3+4i)^{m+1} - 5^{(n+1)/2} }{20} 
 \hfill & \text{ if$\{a, b\} = \{1,3\}$ or $\{2,4\}$};
\end{cases}\]

\end{enumerate}
where $\Re$ (resp. $\Im$) stands for the real (resp. imaginary) part of a complex number. 
\end{cor}

\begin{proof} This follows from previous theorem and Theorem \ref{thmlink}. 
\end{proof}


\section{Application on irreducible polynomials}

Given $a,b \in \mathbb{F}_{q}^{*}$, remember from the introductory section that $P_n(a,b)$ stands for the number of irreducible polynomials of the following form 
$$ T^n - a T^{n-1} + \cdots + (-1)^{n}b \in \Fq[T].$$ 
In \cite[Thm. 5.1]{wan} Wan obtains the following bound for $P_n(a,b)$
\begin{equation}\label{wanbound}
\Big| P_n(a,b) - \frac{q^{n-1}}{n(q-1)} \Big| \leq \tfrac{3}{n}\; q^{\frac{n}{2}}.
\end{equation} 
A connection between $P_n(a,b)$ and $N_n(a,b)$ is establish by Moisio in the following.

\begin{lemma}{\cite[Lemma 2.1]{moisio}} \label{lemma-PnNn} Let $a,b \in \Fq$, then 
\[P_n(a,b) =  \frac{1}{n} \sum_{t | n} \mu(t) N_{n/t}(a,b)
 \]
 where $\mu$ stands for the M\"obius function. 
\end{lemma}

Applying the Katz bound \eqref{katzbound} for $a,b \in \Fq^{*}$, Moisio obtains (cf. \cite[Cor. 3.4 and Cor. 4.3]{moisio}) the following bound for $P_n(a,b)$
\begin{equation}\label{moisiobound}
\Big| P_n(a,b) - \frac{q^{n}-1}{nq(q-1)} \Big| \leq 
q^{(n-2)/2} + \frac{q^{n/2}-1}{q(q-1)} + \tfrac{n}{2}q^{(n-4)/4} 
< \tfrac{2}{q-1}\; q^{\frac{n}{2}}.
\end{equation}
which improves the Wan bound \eqref{wanbound} when $n< \tfrac{3}{2}(q-1).$

\begin{thm} Let $a,b \in \Fq^{*}$, then 
\begin{align*}
\Big| P_n(a,b) - \frac{q^{n}-1}{nq(q-1)} \Big| & \leq 
\frac{1}{n} \left( \frac{1+ (q-2)q^{(n-1)/2} - q^{(n-2)/2}(q^{1/2}-1)(d-1)}{q-1} + \frac{1}{q} \right) \\ 
& \quad + \frac{q^{n/2}-1}{q(q-1)} + \tfrac{n}{2}q^{(n-4)/4}.
\end{align*}
where $d := \gcd(n, q-1).$
\end{thm}

\begin{proof} Let $n = p_{1}^{e_1} \cdots p_{m}^{e_m}$ be the prime number decomposition of $n$ with $p_1 < \cdots < p_m$ and $n'=p_1 \cdots p_m$. From \cite[Lemma 2.2]{moisio} yields
\[N_n(a,b) - \frac{M_1 n'}{2} \leq n P_n(a,b) \leq N_n(a,b) + \frac{ M_2(n'-2)}{2}\]
where $M_1 := \max_{r}\big\{N_{\frac{n}{r}}(a,b)\big\}$ and  $M_2 := \max_{s}\big\{N_{\frac{n}{s}}(a,b)\big\}$ and $r > 1$ (resp. $s > 1$) runs over the factors of $n'$ having odd (resp. even) number of prime factors. If $m=1$, set $M_2 =1$.    

As in the proof of \cite[Cor. 4.3]{moisio}
\[M_2 < M_1 \leq \frac{2(q^{n/2}-1)}{q(q-1)} + n  q^{(n-4)/4}.\]
Therefore,
\begin{align*}
\Big| P_n(a,b) - \frac{q^{n}-1}{nq(q-1)} \Big| & \leq 
\frac{1}{n} \Big| N_n(a,b) - \frac{q^{n-1}-1}{(q-1)} \Big| 
+ \frac{1}{nq} +   \frac{q^{n/2}-1}{q(q-1)} + \tfrac{n}{2}q^{(n-4)/4},
\end{align*}
and the proof follows from Corollary \ref{cor-ASbound2}.
\end{proof}

\begin{rem} In general, our estimate for $P_n(a,b)$ improves the Wan bound \eqref{wanbound} unless that either $q=2$ or $q=3$. Moreover, it improves the Moisio  bound \eqref{moisiobound} when, roughly speaking, $n \geq \sqrt{q}.$
\end{rem}

Next we apply the computation from previous section to give explicit formulas for 
$P_n(a,b)$ for $q=2,3,4,5$. 


\begin{prop} Let $P_n(1,1)$ denote the number of irreducible polynomials of the following form 
$ T^n - T^{n-1} + \cdots + (-1)^{n} \in \mathbb{F}_2[T]$, for some $n \geq 2$.
Then
\[P_n(1,1) = \frac{1}{n} \sum_{t |n} \mu(t) 2^{n/t -1} 
= \frac{1}{n} \sum_{t | n} \mu(\tfrac{n}{t}) 2^{t-1}\]
where $\mu$ denotes the M\"obius function.
\end{prop}

\begin{proof} The proof follows from Lemma \ref{lemma-PnNn}, Proposition \ref{prop-q=2} and the M\"obius inversion formula \cite[Thm. 3.24]{lidl}.
\end{proof}

\begin{cor} Let $P_n(1,1)$ denote the number of irreducible polynomials of the following form 
$ T^n - T^{n-1} + \cdots + (-1)^{n} \in \mathbb{F}_2[T]$, for some $n \geq 2$. If $n$ is a prime number, then
\[ \Big| P_n(a,b) - \frac{2^{n-1}}{n} \Big|  \leq  \frac{1}{n}\]
which improves the Wan bound \eqref{wanbound} for $q=2$.
\end{cor}

\begin{proof} This follows from previous proposition and Lemma \ref{lemma-PnNn}.
\end{proof}


\begin{prop} Let $P_n(a,b)$ denote the number of irreducible polynomials of the following form 
$ T^n - a T^{n-1} + \cdots + (-1)^{n} b \in \mathbb{F}_3[T]$, with $ab \neq 0$ and $n \geq 2$.
Then
\[P_n(a,b) = \frac{1}{n} \sum_{t |n} \mu(n/t) \cdot
\begin{cases} 
  \frac{ 3^t + (-3)^{(t+1)/2} \eta(a)}{6}  \hfill & \text{ if $t$ is odd and }  b =1 \\[4pt]
\frac{ 3^t - (-3)^{(t+1)/2} \eta(a)}{6}  \hfill & \text{ if $t$ is odd and } b =-1   \\[4pt]
\frac{ 3^t + (-3)^{m} \eta(a)}{6}  \hfill & \text{ if $t = 2m$ is even},
\end{cases}\]
where $\mu$ denotes the M\"obius function and $\eta$ is the quadratic character on $\mathbb{F}_{3^n}$.
\end{prop}

\begin{proof} The proof follows from Lemma \ref{lemma-PnNn}, Corollary \ref{cor-q=3} and the M\"obius inversion formula \cite[Thm. 3.24]{lidl}.
\end{proof}


\begin{prop} Let $P_n(a,b)$ denote the number of irreducible polynomials of the following form 
$ T^n - a T^{n-1} + \cdots + (-1)^{n} b \in \mathbb{F}_4[T]$, with $ab \neq 0$ and $n \geq 2$.
Then
\[P_n(a,b) = \frac{1}{n} \sum_{t |n} \mu(n/t) \cdot
\begin{cases} 
\frac{4^t + (-1)^n 2^{t+1}}{12}  \hfill & \text{if $3\mid t$, $t$ is odd and }  b =1 \\[6pt]
\frac{4^t + (-1)^{t-1} 2^{t+2}}{12} 
\hfill & \text{if $3 \nmid t$ and either $a=b=1$ or $\{a, b\} = \{w,w^2\}$,} 
\\[6pt]
\frac{4^t + (-1)^{t} 2^{t+1}}{12}  
\hfill & \text{if $3 \nmid t$ and previous condictions does not hold};
\end{cases}\]  
where $\mu$ denotes the M\"obius function and $w \in \mathbb{F}_{4}$ is a primitive element such that $\mathbb{F}_{4} = \mathbb{F}_2(w)$.
\end{prop}

\begin{proof} The proof follows from Lemma \ref{lemma-PnNn}, Corollary \ref{cor-q=4} and the M\"obius inversion formula \cite[Thm. 3.24]{lidl}.
\end{proof}


\begin{prop} Let $P_n(a,b)$ denote the number of irreducible polynomials of the following form 
$ T^n - a T^{n-1} + \cdots + (-1)^{n} b \in \mathbb{F}_5[T]$, with $ab \neq 0$ and $n \geq 2$. 
Then 
\[
 P_n(a,b) = \frac{1}{n} \sum_{t |n} \mu(\frac{n}{t})\; \varepsilon(t,a,b) \]
where $\varepsilon(t,a,b)$ equals to 
\begin{eqnarray*}
 \frac{ 5^t + 5^m \big( 5^{m} - 2 \cdot\Re(-3+4i)^m \big)}{20} 
 &\text{ if $t = 4m$} \\[6pt] 
 \frac{ 5^t +  5^{m+1} \cdot 2\;\Re(-3+4i)^m + 5^{(t+1)/2}}{20} 
 &\text{\small if $t=4m+1$ and  $a=b=1,4$ or $\{a, b\} = \{2,3\}$} \\[6pt] 
\frac{  5^t -  5^{m+1} \cdot 2\; \Re(-3+4i)^m + 5^{(t+1)/2}}{20}
 & \text{\small if $t=4m+1$ and $a=b=2, 3$ or $\{a, b\} = \{1,4\}$} \\[6pt] 
\frac{ 5^t - 5^{m+1} \cdot 2\; \Im(-3+4i)^m - 5^{(t+1)/2}}{20}
 & \text{\small if $t=4m+1$ and $\{a, b\} = \{1,2\}$ or $ \{3,4\}$} \\[6pt] 
\frac{ 5^t + 5^{m+1} \cdot 2 \;\Im(-3+4i)^m - 5^{(t+1)/2}}{20} 
& \text{\small if $t=4m+1$ and  $\{a, b\} = \{1,3\}$ or $\{2,4\}$} \\[6pt]
\frac{ 5^t + 5^{m+1} \cdot 2\;\Re(-1-2i)(-3+4i)^m + 5^{t/2} }{20}
 & \text{ if $t=4m+2$ and $a= \pm b$}\\[6pt]
\frac{ 5^t - 5^{m+1} \cdot 2\;\Re(-1-2i)(-3+4i)^m + 5^{t/2} }{20}
\hfill & \text{ if $t=4m+2$ and  $a\neq \pm b$}\\[6pt]
\frac{ 5^t - 5^{m+1} \cdot 2\;\Re(-3+4i)^{m+1} + 5^{(t+1)/2} }{20}
 & \text{\small if $t=4m+3$ and $a=b=1,4$ or $\{a, b\} = \{2,3\}$} \\[6pt]
 \frac{  5^t + 5^{m+1} \cdot 2\; \Re(-3+4i)^{m+1} + 5^{(t+1)/2} }{20}
 & \text{\small if $t=4m+3$ and $a=b=2, 3$ or $\{a, b\} = \{1,4\}$} \\[6pt]
\frac{ 5^t + 5^{m+1} \cdot 2\; \Im(-3+4i)^{m+1} - 5^{(t+1)/2} }{20}
 & \text{\small if $t=4m+3$ and $\{a, b\} = \{1,2\}$ or $ \{3,4\}$} \\[6pt]
\frac{ 5^t - 5^{m+1} \cdot 2 \;\Im(-3+4i)^{m+1} - 5^{(t+1)/2} }{20} 
 & \text{\small if $t=4m+3$ and $\{a, b\} = \{1,3\}$ or $\{2,4\}$},
\end{eqnarray*}
$\mu$ denotes the M\"obius function and $\Re$ (resp. $\Im$) stands for the real (resp. imaginary) part of a complex number. 
\end{prop}

\begin{proof} The proof follows from Lemma \ref{lemma-PnNn}, Corollary \ref{cor-q=5} and the M\"obius inversion formula \cite[Thm. 3.24]{lidl}.
\end{proof}


\end{document}